\documentclass
[11pt, a4paper]{article}

\usepackage{inputenc} 
\usepackage[english]{babel}
\usepackage{dsfont}
\usepackage{authblk}

\usepackage[totalheight=24 true cm, totalwidth=15 true cm]{geometry}
\usepackage{enumitem}

\setlist[itemize]{noitemsep}
\setlist[enumerate]{noitemsep}
\usepackage{bbding}
\usepackage{amsmath,amsthm,amssymb}
\usepackage{mathrsfs}
\usepackage{wasysym}
\usepackage{mathtools}
\usepackage{stmaryrd}
\usepackage{url}
\usepackage{wrapfig}
\usepackage{starfont}
\usepackage{pifont}
\usepackage{eurosym}
\usepackage{subcaption}
\usepackage{setspace}
\usepackage{bbm}
\usepackage{dsfont}
\usepackage{url}
\usepackage{amsmath,blkarray,booktabs, bigstrut}
\usepackage{multirow,multicol}
\usepackage{xcolor}
\usepackage{color}
\usepackage{imakeidx}
\makeindex[]

\usepackage{algorithm}
\usepackage{algpseudocode}
\algrenewcommand\algorithmicrequire{\textbf{Input:}}
\algrenewcommand\algorithmicensure{\textbf{Output:}}
\usepackage{relsize}
\usepackage{float}
\usepackage{tikz,pgf}
\usepackage{tikz-cd}
\usepackage{wasysym}
\usepackage{graphicx}
\usepackage{fancyhdr}
\usepackage{stmaryrd}
\usepackage{verbatim}
\usepackage{pgfplots}
\pgfplotsset{compat=1.15}
\usetikzlibrary{arrows}
\usepackage{xspace}

\usepackage[all,dvips,knot,web,arc,curve,color,frame]{xy}
\usepackage[all,knot,arc,color,web]{xy}
\xyoption{arc}
\xyoption{web}
\xyoption{curve}

\numberwithin{equation}{section}

\theoremstyle{plain}
\newtheorem{theorem}{Theorem}[section]
\newtheorem{proposition}[theorem]{Proposition}
\newtheorem{corollary}[theorem]{Corollary}
\newtheorem{lemma}[theorem]{Lemma}
\newtheorem{notation}[theorem]{Notation}

\theoremstyle{definition}
\newtheorem{definition}[theorem]{Definition}

\newtheorem{remark}[theorem]{Remark}

\newtheorem{example}[theorem]{Example}
\newtheorem{question}{Question}

\usepackage[bookmarksnumbered=true]{hyperref} 
\hypersetup{
     colorlinks = true,
     linkcolor = blue,
     anchorcolor = blue,
     citecolor = teal,
     filecolor = blue,
     urlcolor = blue
     }
\newcommand\restr[2]{{
  \left.\kern-\nulldelimiterspace 
  #1 
  \vphantom{\big|} 
  \right|_{#2} 
  }}

\everymath{\displaystyle}
\allowdisplaybreaks

\usepackage{graphicx}

\newcommand{\size}[1]{\left| #1 \right|} 

\newcommand{\ORmn}{\textup{OR}_{d}(K_{m,n})}

\newcommand{\Fr}{\textup{Fr}}
\newcommand{\Real}{\textup{Real}}

\newcommand*\closure[1]{\overline{#1}}

\renewcommand{\tilde}[1]{\widetilde{#1}}

\newcommand{\vv}{\mathbf{v}}

\newcommand{\OR}{\textup{OR}_{d}(\closure{G})}
\newcommand{\3}{\textup{OR}_{3}(\closure{G})}
\newcommand{\4}{\textup{OR}_{4}(\closure{G})}

\newcommand{\Ugn}{U_{G_{[n]\setminus S}}}

\DeclareMathOperator{\rank}{rk}

\DeclareMathOperator{\im}{im}

\DeclareMathOperator{\sspan}{span}

\DeclareMathOperator{\CC}{\mathbb{C}}

\hypersetup{
    colorlinks=true,
    linkcolor=blue,
    filecolor=magenta,      
    urlcolor=cyan,
    pdftitle={Overleaf Example},
    pdfpagemode=FullScreen,
    }

\newtheorem{theoremA}{Theorem}

\newtheorem{theoremB}{Theorem}

\newtheorem{theoremC}{Theorem}

\usepackage{tocloft}
\title{\vspace*{-2em} Lovász--Saks--Schrijver Ideals and the Irreducible Components of the Variety of Orthogonal Representations of a Graph}

\author{Emiliano Liwski}

\affil{\textsubscript{Department of Mathematics, KU Leuven, Celestijnenlaan 200B, B-3001 Leuven, Belgium}}

\date{\vspace*{-.5em}\today\vspace*{-1.5em}}

\begin{document}
\maketitle

\begin{abstract}
Given a finite simple graph $G$ and a positive integer $d$, one can associate to $G$ the Lovász--Saks--Schrijver ideal $L_{G}(d)$, an ideal generated by quadratic polynomials coming from orthogonality conditions. The corresponding variety $\mathbb{V}(L_{G}(d))$, denoted $\OR$, is the variety of orthogonal representations of the complement graph $\overline{G}$: its points are maps from the vertex set of $G$ to $\mathbb{K}^{d}$ that send adjacent vertices of $G$ to orthogonal vectors. 
In this paper we study the irreducible decomposition of $\OR$ and the primary decomposition of $L_{G}(d)$. Our main focus is the case in which $G$ is a forest. Under this assumption, we determine the irreducible components of $\OR$, compute their dimensions, and describe their defining equations, thereby obtaining the primary decomposition of $L_{G}(d)$. The key ingredient is a matroid-theoretic framework in which we associate to every forest $G$ a paving matroid $\mathcal{M}(G)$.
\end{abstract}



\bigskip
\noindent\textbf{\bf Acknowledgement.}  
This work was supported by the PhD fellowship 1126125N and partially funded by the FWO grants G0F5921N (Odysseus), G023721N, and the KU Leuven grant iBOF/23/064.

{\hypersetup{linkcolor=black}
\setcounter{tocdepth}{2}
\setlength\cftbeforesecskip{1.1pt}
\tableofcontents}

\section{Introduction}

\paragraph{Motivation.}

Orthogonal representations of graphs were introduced by Lov\'asz in 1979 \cite{lovasz1979shannon} as a tool for studying the Shannon capacity of graphs. Subsequent works have shown that such representations are closely linked to several fundamental combinatorial properties of graphs. An orthogonal representation assigns to each vertex of a graph $G$ a vector in $\mathbb{K}^{d}$ such that vectors corresponding to non-adjacent vertices are orthogonal.

The set of all orthogonal representations of a graph $G$ on $[n]$ forms an algebraic variety in $\mathbb{K}^{nd}$. Its defining equations encode the orthogonality constraints and give rise to the \emph{Lovász--Saks--Schrijver ideals} (LSS ideals). The study of these ideals lies at the intersection of commutative algebra and combinatorics. As discussed in \cite{conca2019lovasz}, they are also naturally connected to classical topics such as coordinate sections of determinantal varieties.

\paragraph{LSS ideals.}

We recall the definition of Lovász--Saks--Schrijver ideals, following the conventions of \cite{conca2019lovasz}, where points on the variety of the LSS ideal correspond to orthogonal representations of the \emph{complement} of the graph. Recall that the complement of a graph $G$ is the graph $\overline{G}$ on the
same vertex set, where $\{i,j\}\in E(\overline{G})$ if and only if
$\{i,j\}\notin E(G)$.

\begin{definition}\label{def: LSS ideals}
Let $d$ be a positive integer.
Given a simple graph $G=([n],E)$ with vertex set $[n]=\{1,\ldots,n\}$, consider the $n\times d$ matrix of indeterminates $X=(x_{i,j})$, and let $\mathbb{K}[X]$ denote the corresponding polynomial ring.
For each edge $e=\{i,j\}\in E$, define
\begin{equation}\label{polynomials fe}
f_{e} := \sum_{k=1}^{d} x_{i,k} x_{j,k}.
\end{equation}
The \emph{Lovász--Saks--Schrijver ideal} (or \emph{LSS ideal}) of $G$ in dimension $d$ is
\[
L_{G}(d) = (f_{e}: e\in E)\subset \mathbb{K}[X].
\]
We write $\OR := \mathbb{V}(L_{G}(d))\subset \mathbb{K}^{nd}$ for the associated variety.
A point $\mathbf{v}\in \OR$ corresponds to an $n$-tuple of vectors $(v_{1},\ldots,v_{n})\in(\mathbb{K}^{d})^{n}$ satisfying the orthogonality relations $v_{i}\cdot v_{j}=0$ whenever $\{i,j\}\in E$.
\end{definition}

The foundational work on $L_{G}(d)$ and $\OR$ is due to \cite{lovasz1989orthogonal}. In that work, the authors studied the subvariety (possibly empty) given by the closure of orthogonal representations in \emph{general position}, that is, representations in which any $d$ vectors are linearly independent. They proved that this subvariety is irreducible and nonempty precisely when $\overline{G}$ is $(n-d)$-connected.

\paragraph{Main question.}

The central question addressed in this work is the following.

\begin{question}\label{main question}
Determine the irreducible decomposition of $\OR$ and the primary decomposition of $L_{G}(d)$ when $G$ is a forest and $d\geq 3$.
\end{question}

Throughout, we work over the field $\mathbb{K}=\mathbb{C}$, assume $d\geq 3$, and take $G$ to be a forest. We answer Question~\ref{main question} over $\mathbb{C}$ and compute the dimensions of the irreducible components of $\OR$.

\paragraph{Related work.}

When $d=1$, the ideal $L_{G}(d)$ reduces to the monomial \emph{edge ideal} of $G$, a well-studied object (see \cite{morey2012edge}).
For $d=2$, the primary decomposition and radicality of $L_{G}(2)$ were determined in \cite{herzog2015ideal}: the ideal is radical over fields of characteristic $\neq 2$, and over fields of characteristic $2$ it is radical exactly when $G$ is bipartite.
If the field contains $\sqrt{-1}$, then $L_{G}(2)$ becomes isomorphic to either the permanental edge ideal $\Pi_{G}$ \cite[Section~3]{herzog2015ideal} or, when $G$ is bipartite, the binomial edge ideal of $G$ \cite[Remark~1.5]{herzog2015ideal}.

Parity binomial edge ideals, studied in \cite{kahle2016parity}, coincide with LSS ideals in characteristic $\neq 2$ after a change of variables; their minimal primes and mesoprimary decompositions are described there.

The work \cite{conca2019lovasz} provided the first examples of non-radical ideals $L_{G}(d)$ for $d>2$, introduced the invariant $\mathrm{pmd}(G)$, and showed that $L_{G}(d)$ is a radical complete intersection for $d\geq\mathrm{pmd}(G)$ and prime for $d\geq\mathrm{pmd}(G)+1$. They further related LSS ideals to coordinate sections of determinantal ideals.

For forests, \cite{conca2019lovasz} proved that $L_{G}(d)$ is prime whenever $d$ is greater than the maximum degree of a vertex of $G$, resolving the decomposition problem for sufficiently large $d$.
Our work focuses instead on determining the \emph{primary decomposition} of $L_{G}(d)$ for all $d\geq 3$ over $\mathbb{C}$.

Several additional papers explore algebraic properties of LSS ideals: complete intersection and primeness, regularity, normality, Gr\"obner bases, and connections with tensors; see
\cite{kapon2019singularity, tolosa2025normality, gharakhloo2023hypergraph, kumar2021lovasz, amalore2025almost, dg2025grobner}.

\paragraph{Outline of the strategy.} 
Let $G$ be a forest graph on $[n]$. Our approach relies on two constructions associated to $G$:  
(1) the varieties $V_{S}$ for subsets $S\subseteq [n]$, and  
(2) the paving matroid $\mathcal{M}(G)$, defined below.

\begin{definition}[Definition~\ref{def: US and VS}]
For $S\subseteq [n]$, set
\[
U(S):=\Bigl\{\mathbf{v}=(v_{1},\ldots,v_{n})\in \OR \,\Big|\, v_{i}=0 \ \text{if and only if}\ i\in S\Bigr\}\subset (\CC^{d})^{n},
\]
and let $V_{S}:=\overline{U(S)}\subset \CC^{dn}$ be its Zariski closure.
\end{definition}

As shown in Proposition~\ref{prop: redundant decomposition}, each $V_{S}$ is irreducible, and the union $\cup_{S\subseteq [n]} V_{S}$ gives an irreducible but redundant decomposition of $\OR$. To extract its maximal components, we introduce a matroid associated to $G$.

\begin{definition}[Definition~\ref{def: MG}]
The matroid $\mathcal{M}(G)$ is the rank-$d$ paving matroid whose dependent hyperplanes are exactly the sets
\[
N_{G}(i)\qquad\text{for all }i\in[n]\text{ with }|N_{G}(i)|\ge d,
\]
where $N_{G}(i)$ denotes the neighborhood of $i$ in $G$.
\end{definition}

The key interaction between the varieties $V_{S}$ and the matroid $\mathcal{M}(G)$ is the following: for a generic point $\mathbf{v}=(v_{1},\ldots,v_{n})\in V_{S}$, the realizable matroid associated to the nonzero vectors $(v_{i}:i\notin S)$—that is, the matroid whose dependent sets reflect the linear dependencies among them—is precisely $\mathcal{M}(G_{[n]\setminus S})$; see Lemma~\ref{lem: is dense}. This allows us to identify the maximal varieties in the redundant decomposition: they correspond exactly to the admissible subsets introduced in Definition~\ref{def: admissible}.

\medskip
\noindent
\textbf{\large Our contributions.}  
We now summarize the main results of this paper.  
All statements concern the irreducible decomposition and the dimensions of the components of $\OR$ for a forest $G$, as well as the primary decomposition of the associated LSS ideal $L_{G}(d)$.

\begin{theoremA}
Let $G$ be a forest on $[n]$. The irreducible decomposition of $\OR$ is
\[
\OR = \bigcup_{S} V_{S}, \tag{Theorem~\ref{main theorem}}
\]
where the union ranges over all $G$-admissible subsets of $[n]$, that is, all subsets $S\subseteq [n]$ such that for every $i\in S$,
\[
\bigl|N_G(i)\cap([n]\setminus S)\bigr| \ge d
\quad\text{and}\quad
N_G(i)\cap([n]\setminus S)\not\subseteq 
N_G(j)\cap([n]\setminus S)
\ \text{for every } j\in [n]\setminus S,
\]
as in Definition~\ref{def: admissible}.
\end{theoremA}

We also determine the dimensions of the irreducible components of $\OR$.

\begin{theoremB}
Let $S$ be a $G$-admissible subset of $[n]$. Then
\[
\dim(V_{S}) \;=\; d(n-\lvert S\rvert)\;-\; \lvert E(G_{[n]\setminus S})\rvert.
\tag{Theorem~\ref{thm: dim of components}}
\]
\end{theoremB}

Finally, we obtain the primary decomposition of the LSS ideal $L_{G}(d)$.

\begin{theoremC}
Let $G$ be a forest. Then the primary decomposition of $L_{G}(d)$ is
\[
L_{G}(d)
\;=\;
\bigcap_{S}
\ \bigl( (x_{i,j} : i \in S)\;+\;\sqrt{I_{G_{[n]\setminus S}}} \bigr),\tag{Theorem~\ref{thm: primary decomposition}}
\]
where the intersection ranges over all $G$-admissible subsets of $[n]$.  
For a forest $G$, the ideal $I_{G}$ is defined in Definition~\ref{gm}.
\end{theoremC}

\medskip
\noindent
\textbf{Outline of the paper.}
Section~\ref{section 2} presents the necessary background on matroid theory used throughout the article.  
In Section~\ref{section 3} we introduce the varieties $V_S$ and the paving matroid $\mathcal{M}(G)$, and use them to describe the irreducible components of $\OR$ when $G$ is a forest.  
Section~\ref{section 4} is devoted to computing the dimensions of these components.  
In Section~\ref{section 5} we determine the defining ideals of the varieties $V_S$ and derive the primary decomposition of the LSS ideal $L_G(d)$.  
Finally, Section~\ref{section 6} applies the results of the preceding sections to several classical families of forest graphs, providing explicit descriptions of the irreducible components of $\OR$ and their dimensions.

\section{Preliminaries}\label{section 2}

We begin by recalling several fundamental notions from matroid theory that will be used throughout this work. For a comprehensive introduction, we refer the reader to~\cite{Oxley}. 

\begin{definition}
A \emph{matroid} $M$ consists of a ground set $[n]$ together with a collection $\mathcal{I}(M)$ of subsets of $[n]$, called \emph{independent sets}, satisfying:
\begin{enumerate}[label=(\roman*)]
\item $\emptyset \in \mathcal{I}$;
\item if $I \in \mathcal{I}$ and $I' \subseteq I$, then $I' \in \mathcal{I}$;
\item if $I_1, I_2 \in \mathcal{I}$ with $|I_1| < |I_2|$, then there exists $e \in I_2 \setminus I_1$ such that $I_1 \cup \{e\} \in \mathcal{I}$.
\end{enumerate}
A subset of $[n]$ that is not independent is called \emph{dependent}. For any $F \subseteq [n]$, its \emph{rank}, denoted $\rank(F)$, is the size of the largest independent subset of $F$. The \emph{rank of $M$} is $\rank([n])$.
\end{definition}

A particularly important class of matroids that we will use in this work is the family of paving matroids. It was conjectured in~\cite{mayhew2011asymptotic} that asymptotically almost all matroids belong to this class.

\begin{definition}\normalfont \label{pav}
A matroid $M$ of rank $d$ is called a \emph{paving matroid} if every $(d-1)$-subset of the ground set is independent. A \emph{dependent hyperplane} of $M$ is a maximal dependent set in which every subset of size $d$ is dependent.
\end{definition}

The following lemma from~\cite{hartmanis1959lattice} provides a hypergraph-theoretic characterization of paving matroids.

\begin{lemma}\label{sub h}
Let $d$ and $n$ be integers with $n \geq d + 1$. Suppose $\mathcal{L}$ is a collection of subsets of $[n]$, which satisfies $|l| \geq d$ for all $l\in \mathcal{L}$. Then $\mathcal{L}$ forms the collection of dependent hyperplanes of a paving matroid of rank $d$ on $[n]$ if and only if
\begin{equation}\label{condition for paving}
|l_1 \cap l_2| \leq d - 2 \qquad \text{for all distinct } l_1, l_2 \in \mathcal{L}.
\end{equation}
\end{lemma}

\begin{example}
Consider the collection
\[
\mathcal{L}
=\{\{1,2,4\},\{1,3,7\},\{1,5,6\},\{2,3,5\},\{2,6,7\},\{3,4,6\},\{4,5,7\}\}.
\]
For $d = 3$, one easily verifies that $\mathcal{L}$ satisfies condition~\eqref{condition for paving}. Hence $\mathcal{L}$ gives the family of dependent hyperplanes of a paving matroid of rank $3$ on $[7]$, known as the \emph{Fano plane}.
\end{example}


\section{\texorpdfstring{Irreducible decomposition of $\OR$ for forest graphs}{Irreducible decomposition of OR for forest graphs}}\label{section 3}

In this section we describe the irreducible components of $\OR$ in the case where $G$ is a forest and $d\geq 3$, assumptions that will remain fixed throughout this paper. Recall that we view each point $\mathbf{v}\in \OR$ as an $n$-tuple of vectors 
$(v_{1},\ldots,v_{n})\in (\mathbb{C}^{d})^{n}$ satisfying the orthogonality
relations $v_{i}\cdot v_{j}=0$ whenever $\{i,j\}\in E(G)$.

\subsection{Variety $U_{G}$}

We begin by introducing a subvariety $U_{G}$ that will play a main role in the discussion.

\begin{definition}
We define
\[
U_{G}:=\{\mathbf{v}=(v_{1},\ldots,v_{n})\in \OR : v_{i}\neq 0 \ \text{for every $i$}\}\subset \CC^{nd}.
\]
This is a quasi-affine variety, that is, the intersection of a closed set with an open set, since it is determined by the vanishing conditions defining $\OR$ together with the non-vanishing constraints $v_{i}\neq 0$.
\end{definition}

We next show that $U_{G}$ is nonempty and irreducible.

\begin{lemma}\label{lem: nonempty and irreducible}
$U_{G}$ is nonempty and irreducible.
\end{lemma}

\begin{proof}
We argue by induction on the number of vertices $n$. The base case $n=1$ is clear. Assume the statement holds for all forests with at most $n-1$ vertices. Let $G$ be a forest on the vertex set $[n]$.

Since $G$ is a forest, there exists a vertex $i \in [n]$ with degree at most one. Consider the induced subgraph $G_{[n]\setminus \{i\}}$, which is again a forest on fewer vertices. By the inductive hypothesis, $U_{G_{[n]\setminus \{i\}}}$ is nonempty and irreducible. We distinguish two cases:

\medskip
\noindent \textbf{Case 1:} $\deg(i) = 0$. In this case, the vertex $i$ imposes no orthogonality conditions. Thus, the vector $v_i$ can be any nonzero vector in $\CC^d$, and we have an isomorphism
\begin{equation}\label{isomorphism}
U_{G} \cong (\CC^d)^* \times U_{G_{[n]\setminus \{i\}}}.
\end{equation}
It follows that $U_G$ is nonempty and irreducible.

\medskip
\noindent \textbf{Case 2:} $\deg(i) = 1$. Let $j \in [n]$ be the unique vertex adjacent to $i$. To show nonemptiness, take any
\[
\mathbf{v} = (v_1, \ldots, v_{i-1}, v_{i+1}, \ldots, v_n) \in U_{G_{[n]\setminus \{i\}}}.
\]
Extending $\mathbf{v}$ by choosing $v_i$ as any nonzero vector orthogonal to $v_j$ yields a point in $U_G$, proving nonemptiness. For irreducibility, consider the projection map
\begin{equation}\label{fibration}
\pi : U_G \to U_{G_{[n]\setminus \{i\}}}, \quad \mathbf{v} = (v_1, \ldots, v_n) \mapsto (v_1, \ldots, v_{i-1}, \hat{v}_i, v_{i+1}, \ldots, v_n),
\end{equation}
where $\hat{v}_i$ denotes removal of $v_i$. The fiber over a point $\mathbf{v}$ consists of all nonzero vectors $v_i$ in the hyperplane orthogonal to $v_j$, which is isomorphic to $(\CC^{d-1})^{*}$. Hence, $\pi$ is a $(\CC^{d-1})^{*}$-fibration over $U_{G_{[n]\setminus \{i\}}}$. Since the base is irreducible, so is $U_G$.
\end{proof}

\subsection{Decomposition into varieties $V_{S}$}

We now define a variety associated to each subset $S\subseteq [n]$.

\begin{definition}\label{def: US and VS}
For $S\subseteq [n]$, set
\[
U(S):=\Bigl\{\mathbf{v}=(v_{1},\ldots,v_{n})\in \OR \,\Big|\, v_{i}=0 \ \text{if and only if}\ i\in S\Bigr\}\subset (\CC^{d})^{n}.
\]
We denote by $V_{S}:=\closure{U(S)}\subset \CC^{nd}$ the Zariski closure of $U(S)$ in $\CC^{nd}$.
\end{definition}

\begin{example}\label{example: running example}
Consider the forest graph $G$ from Figure~\ref{forest graph} with vertex set $[11]$ and let $d=4$. Consider the subset $S=\{9,10,11\}$. Then, 
\[U(S)=\{\vv=(v_{1},\ldots,v_{11})\in \textup{OR}_{4}(\closure{G}): \text{$v_{i}=0$ for $i\in \{9,10,11\}$ and $v_{i}\neq 0$ for $i\in [8]$}\}.\]
The variety $V_{S}$ is the closure of $U_{S}$ in $\CC^{44}$. 
\end{example}

\begin{figure}[H]
    \centering
    \includegraphics[width=0.5\textwidth]{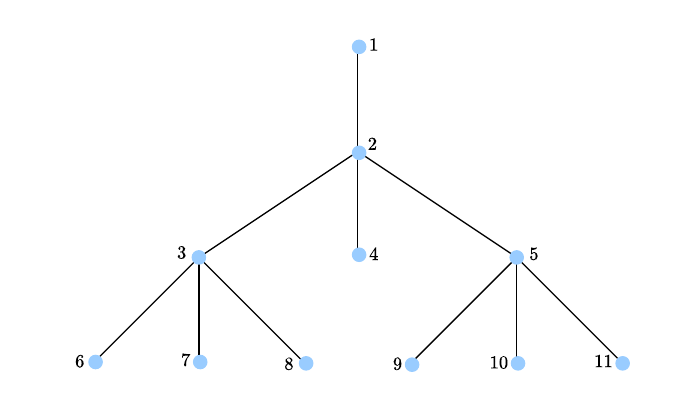}
    \caption{A forest graph.}
    \label{forest graph}
\end{figure}

The next proposition provides an irreducible but redundant decomposition of $\OR$.

\begin{proposition}\label{prop: redundant decomposition}
The variety $\OR$ admits the following irreducible decomposition:
\begin{equation}\label{eq:equality OR}\OR=\bigcup_{S\subseteq [n]}V_{S}.\end{equation}
\end{proposition}

\begin{proof}
The inclusion $\supseteq$ is immediate, since by definition every $U(S)$ is contained in $\OR$. 
For the reverse inclusion, take any $\vv=(v_{1},\ldots,v_{n})\in \OR$ and define
\[
S := \{\, i\in [n] : v_{i}=0 \,\}.
\]
Since $\vv\in \OR$, 
it follows that $\vv\in U(S)$, which proves the equality in~\eqref{eq:equality OR}.

Since each $U(S)$ is isomorphic to $\Ugn$, Lemma~\ref{lem: nonempty and irreducible} shows that $U(S)$ is irreducible, and the same holds for its closure $V_{S}$. 
This establishes the irreducibility of the decomposition in~\eqref{eq:equality OR}.
\end{proof}

\subsection{Matroid $\mathcal{M}(G)$}

In this subsection, we introduce the paving matroid $\mathcal{M}(G)$ associated with any forest graph $G$.  
This matroid will play a central role in identifying the maximal varieties in~\eqref{eq:equality OR}.

\begin{definition}\label{def: MG}
Let $G$ be a forest on the vertex set $[n]$. We associate to $G$ a matroid $\mathcal{M}(G)$ on $[n]$ as follows.  
The matroid $\mathcal{M}(G)$ is the paving matroid of rank $d$ whose dependent hyperplanes are precisely the sets
\[
N_{G}(i) \qquad \text{for all } i\in [n] \text{ with } |N_{G}(i)|\geq d,
\]
where $N_{G}(i)$ denotes the neighborhood of $i$ in $G$. 

Since $G$ is a forest, it contains no $4$-cycle, and therefore any two such subsets intersect in at most one vertex, in particular in at most $d-2$ since $d\geq 3$. Hence, by~\eqref{condition for paving}, this paving matroid is well defined. 
\end{definition}

\begin{example}\label{example: description of MG}
Consider the forest graph $G$ from Example~\ref{example: running example}; see Figure~\ref{forest graph}. 
Let $d=4$. The only vertices of $G$ with degree at least four are $2$, $3$, and $5$. 
Hence the associated matroid $\mathcal{M}(G)$ is the paving matroid of rank four on $[11]$ with dependent hyperplanes
\[
\{N_{G}(2), N_{G}(3), N_{G}(5)\}
    = \bigl\{
        \{1,3,4,5\},
        \{2,6,7,8\},
        \{2,9,10,11\}
      \bigr\}.
\]
\end{example}

\begin{figure}[H]
    \centering
    \includegraphics[width=0.4\textwidth]{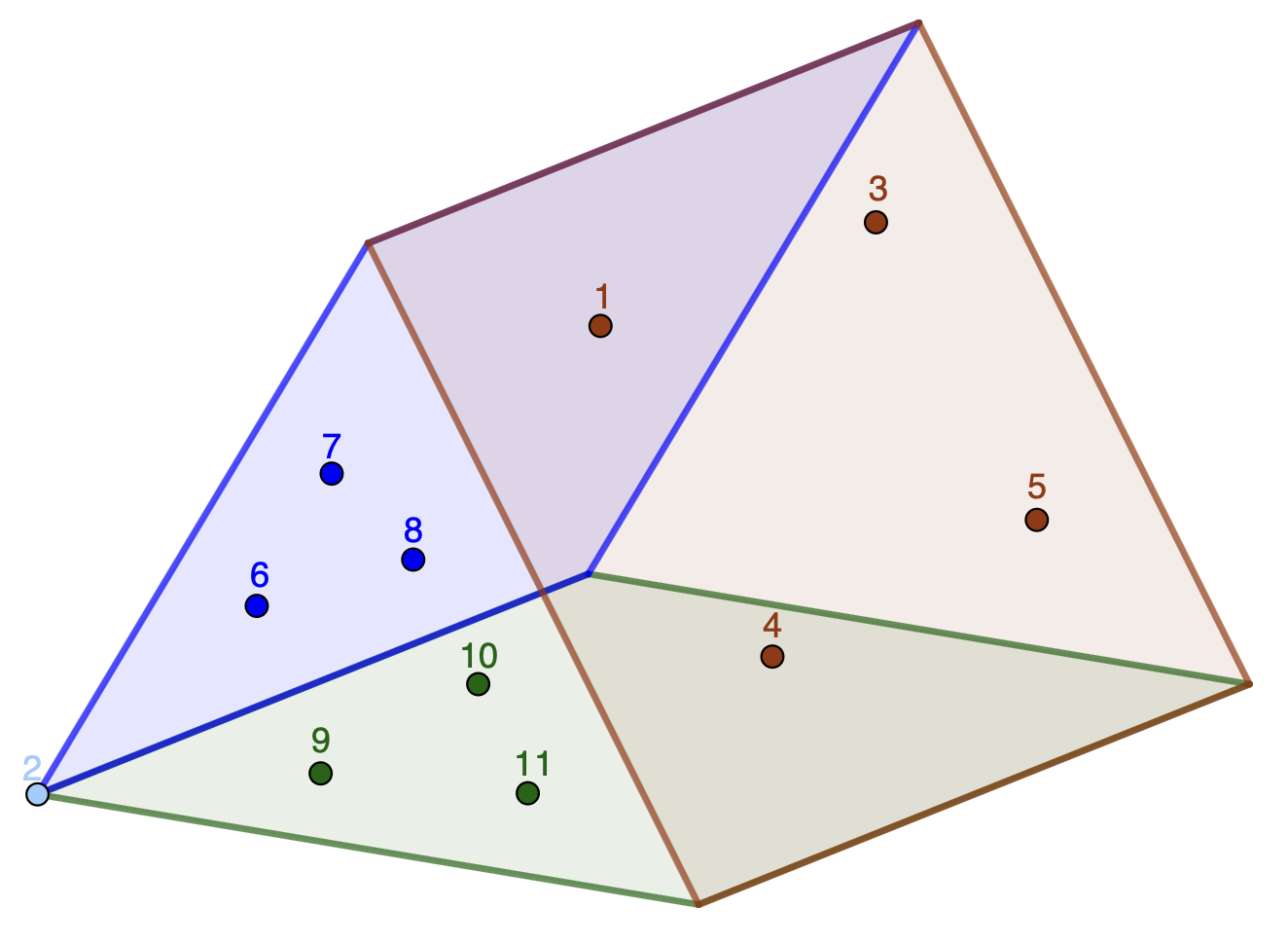}
    \caption{The matroid $\mathcal{M}(G)$ from Example~\ref{example: description of MG} associated with the forest in Figure~\ref{forest graph}.}
    \label{paving matroid MG}
\end{figure}

\begin{notation}
For a tuple of vectors $\vv=(v_{1},\ldots,v_{n})\in(\CC^{d})^{n}$, we write $\mathcal{M}(\vv)$ for the realizable matroid on $[n]$ whose dependent sets are precisely those subsets of $[n]$ corresponding to linear dependencies among the vectors of $\vv$. We say that $\vv$ realizes a matroid $N$ if $\mathcal{M}(\vv)=N$.
\end{notation}

\begin{example}
Consider the matroid $\mathcal{M}(G)$ from Figure~\ref{paving matroid MG} and Example~\ref{example: description of MG}. 
This matroid can be realized by a configuration of $11$ points in $\mathbb{P}^{3}$, as shown in Figure~\ref{paving matroid MG}. 
In this representation, four points lie in a common hyperplane if and only if the corresponding elements of the ground set belong to a common dependent hyperplane in $\mathcal{M}(G)$. From this realization in $\mathbb{P}^{3}$, one can also obtain a realization in $\mathbb{C}^{4}$.
\end{example}

We now observe that the matroid determined by any $\vv\in U_{G}$ contains all the dependencies of $\mathcal{M}(G)$.

\begin{lemma}\label{lem: more dependent}
If $\vv\in U_{G}$, then every dependent set of $\mathcal{M}(G)$ is also dependent in 
$\mathcal{M}(\vv)$.
\end{lemma}

\begin{proof}
It suffices to verify that every dependent hyperplane of $\mathcal{M}(G)$ has rank at most $d-1$ in $\mathcal{M}(\vv)$. 
Let $N_{G}(i)$ be such a hyperplane with $\size{N_{G}(i)}\geq d$. 
Since $\vv\in \OR$, the vectors $\{v_{j}: j\in N_{G}(i)\}$ are all orthogonal to $v_{i}$; and because $v_{i}\neq 0$, they lie in the hyperplane orthogonal to $v_{i}$. 
Consequently, these vectors span a space of dimension at most $d-1$, proving the claim.
\end{proof}

We will also need the following lemma. We say that an $n$-tuple of vectors 
$\mathbf{v}=(v_{1},\ldots,v_{n})\in (\mathbb{C}^{d})^{n}$ is a \emph{faithful orthogonal representation} of $\overline{G}$ if 
$v_{i}\cdot v_{j}=0$ \emph{if and only if} $\{i,j\}\in E(G)$.

\begin{lemma}\label{lem: nice conditions}
Let $G$ be a forest on $[n]$, and let $\mathcal{M}(G)$ be the associated matroid. 
Then there exists a tuple of vectors $\vv=(v_{1},\ldots,v_{n})\in(\CC^{d})^{n}$ such that:
\begin{itemize}
\item $\vv$ is a faithful orthogonal representation of $\overline{G}$;
\item $\vv$ realizes $\mathcal{M}(G)$.
\end{itemize}
\end{lemma}

\begin{proof}
We argue by induction on $n$. The claim is obvious when $n=1$. Assume it holds for all forests with at most $n-1$ vertices, and let $G$ be a forest on $[n]$.

Since $G$ is a forest, there exists a vertex $i\in [n]$ of degree at most one. By the inductive hypothesis, there is a tuple of vectors $(v_{1},\ldots,v_{i-1},v_{i+1},\ldots,v_{n})$ in $\CC^{d}$ which is a faithful orthogonal representation of $\overline{G_{[n]\setminus\{i\}}}$ and realizes $\mathcal{M}(G_{[n]\setminus\{i\}})$. We now extend this tuple by defining $v_{i}$, distinguishing two cases.

\smallskip
\noindent{\bf Case 1: $\deg(i)=0$.}  
Here $i$ imposes no orthogonality conditions. Choose $v_{i}$ as any vector in $\CC^{d}$ which is not orthogonal to any $v_{k}$ and does not lie in the span of any $d-1$ vectors of $\vv$. This yields a tuple $\vv'=(v_{1},\ldots,v_{i-1},v_{i},v_{i+1},\ldots,v_{n})$ satisfying both conditions.

\smallskip
\noindent{\bf Case 2: $\deg(i)=1$.}  
Let $j$ be the unique neighbor of $i$. We must choose $v_{i}$ so that the extended tuple satisfies:
\begin{itemize}
\item[(i)] $v_{i}\in v_{j}^{\perp}$,
\item[(ii)] $v_{i}\notin v_{k}^{\perp}$ for every $k\neq j$,
\item[(iii)] $v_{i}\notin \mathrm{span}(v_{k}:k\in A)$ for every subset $A\subset [n]\setminus\{i\}$ of size $d-1$ with $A\not\subseteq N_{G}(j)$.
\end{itemize}

This conditions would guarantee that the conditions of the statement hold.
Since the tuple $\vv$ realizes $\mathcal{M}(G_{[n]\setminus\{i\}})$, the vectors $v_{k}$ for $k\neq j$ are all distinct from $v_{j}$, and hence the hyperplanes $v_{k}^{\perp}$ for $k\neq j$ are all distinct from $v_{j}^{\perp}$. Hence the set
\[
v_{j}^{\perp}\setminus \bigcup_{k\in [n]\setminus\{i,j\}} v_{k}^{\perp}
\]
is nonempty and Zariski dense in $v_{j}^{\perp}$. Satisfying conditions (i) and (ii) is equivalent to choosing $v_{i}$ as any vector from this set.

For (iii), let $A\subset [n]\setminus\{i\}$ be  a subset of size $d-1$ with $A\not\subseteq N_{G}(j)$, and set $H_{A}:=\mathrm{span}(v_{k}:k\in A)$. We claim that $H_{A}\neq v_{j}^{\perp}$. Indeed, if $H_{A}=v_{j}^{\perp}$, then all $v_{k}$ with $k\in A$ would be orthogonal to $v_{j}$, which, by the faithfulness of the representation $\vv$, would imply $A\subset N_{G}(j)$, a contradiction. Thus $H_{A}\neq v_{j}^{\perp}$.

Consequently, there exist vectors in $v_{j}^{\perp}$ that lie outside every $v_{k}^{\perp}$ with $k\neq j$ and outside every $H_{A}$ as above. Choosing such a $v_{i}$ and extending the tuple $\vv'$ then satisfies the conditions of the lemma, completing the inductive step.
\end{proof}

\subsection{Identifying the maximal varieties $V_{S}$}

In this subsection, we identify the maximal varieties appearing in~\eqref{eq:equality OR}. To this end, we introduce the auxiliary subvarieties $F(S)$.


\begin{definition}
For a subset $S\subseteq [n]$, we define
\[
F(S)\subset U(S)
\]
to be the subvariety formed by all tuples $\vv=(v_{1},\ldots,v_{n})\in(\CC^{d})^{n}$ satisfying:
\begin{itemize}
\item $v_{i}=0$ for every $i\in S$;
\item the restriction $\restr{\vv}{[n]\setminus S}$ is a faithful orthogonal representation of $\overline{G_{[n]\setminus S}}$;
\item the restriction $\restr{\vv}{[n]\setminus S}$ realizes the matroid $\mathcal{M}(G_{[n]\setminus S})$.
\end{itemize}
\end{definition}

\begin{example}
Let $G$, $d$, and $S$ be as in Example~\ref{example: running example}. 
Then the set $F(S)$ consists of all tuples $\mathbf{v} = (v_{1}, \ldots, v_{11}) \in (\mathbb{C}^{4})^{11}$
satisfying the following conditions:
\begin{itemize}
    \item $v_{9} = v_{10} = v_{11} = 0$;
    \item $v_{2} \perp v_{i}$ for all $i \in \{1,3,4,5\}$, 
          $v_{3} \perp v_{i}$ for all $i \in \{2,6,7,8\}$, 
          and $v_{j} \not\perp v_{k}$ for all other pairs $j,k \in [8]$;
    \item the tuple $(v_{1}, \ldots, v_{8})$ realizes the matroid $\mathcal{M}(G_{[8]})$, 
          which is the rank-four paving matroid on $[8]$ with dependent hyperplanes 
          $\{1,3,4,5\}$ and $\{2,6,7,8\}$.
\end{itemize}
\end{example}

We will now see that each $F(S)$ is a dense subset of $V_{S}$.

\begin{lemma}\label{lem: is dense}
$F(S)$ is dense in $V_{S}$.
\end{lemma}

\begin{proof}
We first prove that $F(S)$ is open in $V_{S}$. It is easy to see from the defining conditions of $U(S)$ that $U(S)$ is open in $V_{S}$. It is then enough with proving that $F(S)$ is open in $U(S)$. The faithfulness of $\restr{\vv}{[n]\setminus S}$ is clearly an open condition since is given by the non orthogonality of vectors $v_{j}$ and $v_{k}$ for $\{j,k\}\notin E(G)$ with $j,k\in [n]\setminus S$.

On the other hand, by Lemma~\ref{lem: more dependent}, we know that for each $\vv\in U(S)$, the matroid associated to the restriction $\restr{\vv}{[n]\setminus S}$ cotains among its dependencies the dependencies of $\mathcal{M}(G_{[n]\setminus S})$. Hence, saying that $\restr{\vv}{[n]\setminus S}$ realizes the matroid $\mathcal{M}(G_{[n]\setminus S})$ is equivalent to saying that there are not extra dependencies, which is clearly an open condition in $U(S)$. Thus, $F(S)$ is open in $U(S)$ and thus in $V_{S}$.

Moreover, by Lemma~\ref{lem: nonempty and irreducible}, it follows that $U(S)$ is irreducible and also its closure $V_{S}$, and by Lemma~\ref{lem: nice conditions}, it follows that $F(S)$ is nonempty. Thus, since $F(S)$ is open and nonempty in the irreducible variety $V_{S}$, it follows that is dense there.
\end{proof}

We now describe the subsets $S \subseteq [n]$ that give rise to the maximal varieties appearing in~\eqref{eq:equality OR}. These subsets form the following class.

\begin{definition}\label{def: admissible}
A subset $S \subseteq [n]$ is called \emph{$G$-admissible} if, for every $i \in S$,
\begin{equation}\label{admissible description}
\bigl|N_G(i)\cap([n]\setminus S)\bigr| \geq d
\quad \text{and} \quad
N_G(i)\cap([n]\setminus S)\not\subseteq
N_G(j)\cap([n]\setminus S)
\ \text{for every } j\in [n]\setminus S.
\end{equation}
Note that this condition is equivalent to
\begin{equation}\label{full rank}
\rank\bigl(N_{G}(i)\cap ([n]\setminus S)\bigr)=d
\quad \text{in $\mathcal{M}(G_{[n]\setminus S})$.}
\end{equation}
\end{definition}

\begin{example}\label{example: admissible subsets}
Consider the forest graph $G$ from Example~\ref{example: running example} (see Figure~\ref{forest graph}) and let $d = 4$. 
Since every vertex of $G$ has degree at most $4$, and no two vertices of degree $4$ share the same set of neighbors, condition~\eqref{admissible description} reduces to
\[
\bigl|N_G(i) \cap ([n] \setminus S)\bigr| = 4
\]
for all $i \in S$. 
Consequently, the $G$-admissible subsets are precisely
\[
\bigl\{\emptyset,\, \{2\},\, \{3\},\, \{5\},\, \{3,5\}\bigr\}.
\]
\end{example}

\begin{proposition}\label{prop: maximal components}
Let $S \subseteq [n]$. Then $V_{S}$ is maximal, with respect to inclusion, among the varieties in~\eqref{eq:equality OR} if and only if $S$ is $G$-admissible.
\end{proposition}

\begin{proof}
We first prove that if $S$ is not $G$-admissible, then $V_{S}$ is redundant in~\eqref{eq:equality OR}. 
Since $S$ is not $G$-admissible and by Equation~\eqref{full rank}, there exists a vertex $i\in S$ such that
\begin{equation}\label{not full rank}
\rank\bigl(N_{G}(i)\cap ([n]\setminus S)\bigr)<d
\qquad \text{in $\mathcal{M}(G_{[n]\setminus S})$.}
\end{equation}
We claim that in this situation 
\[
F(S)\subset \overline{U(S\setminus \{i\})},
\]
which implies
\[
V_{S}=\overline{F(S)}\subset \overline{U(S\setminus \{i\})}=V_{S\setminus \{i\}},
\]
where the first equality follows from Lemma~\ref{lem: is dense}.

Take an arbitrary tuple of vectors $\vv=(v_{1},\ldots,v_{n})\in F(S)$. 
By definition, the restriction of $\vv$ to $[n]\setminus S$ realizes the matroid $\mathcal{M}(G_{[n]\setminus S})$. 
Combining this with~\eqref{not full rank}, we obtain
\[
\rank\{v_{k}:k\in N_{G}(i)\cap ([n]\setminus S)\}<d.
\]
Hence, there exists a nonzero vector $w\in \CC^{d}$ orthogonal to all these vectors, i.e.,
\begin{equation}\label{perp}
w\perp v_{k} \ \text{for every $k\in N_{G}(i)\cap ([n]\setminus S)$},
\qquad w\neq 0.
\end{equation}

For each $\lambda\in \CC^{\ast}$, define a new tuple 
$\vv_{\lambda}=(v_{\lambda}^{1},\ldots,v_{\lambda}^{n})$ by
\[
v_{\lambda}^{j}=
\begin{cases}
v_{j}, & j\neq i,\\
\lambda w, & j=i.
\end{cases}
\]
By the orthogonality in~\eqref{perp}, the fact that $\vv\in U(S)$, and since $\lambda w\neq 0$, it follows that $\vv_{\lambda}\in U(S\setminus \{i\})$ for all $\lambda \in \CC^{\ast}$. 
Letting $\lambda \to 0$, we have $\vv_{\lambda}\to \vv$, which shows that $\vv\in \overline{U(S\setminus \{i\})}$. 
As $\vv$ was arbitrary, this proves $F(S)\subset \overline{U(S\setminus \{i\})}$, as claimed. 

\medskip

We now prove the converse: if $S$ is $G$-admissible, then $V_{S}$ is irredundant in~\eqref{eq:equality OR}. 
Assume, for a contradiction, that there exists a subset $S'\neq S$ with $V_{S}\subseteq V_{S'}$. 
Take any $\vv=(v_{1},\ldots,v_{n})\in F(S)\subset U(S)$. 
Since $\vv\in U(S)\subset V_{S'}$, it follows from the definition that $v_{j}=0$ for all $j\in S'$. 
On the other hand, as $\vv\in F(S)$, we have $v_{j}=0$ if and only if $j\in S$, so necessarily $S'\subsetneq S$.

Choose $i\in S\setminus S'$. 
Since $S$ is $G$-admissible and the restriction of $\vv$ to $[n]\setminus S$ is a realization of $\mathcal{M}(G_{[n]\setminus S})$, it follows from~\eqref{full rank} that there exists a subset of $d$ vertices $\{i_{1},\ldots,i_{d}\}\subset N_{G}(i)\cap ([n]\setminus S)$ such that $\{v_{i_{1}},\ldots,v_{i_{d}}\}$ are linearly independent in $\CC^{d}$.

By Lemma~\ref{lem: more dependent}, we have that every dependent set of $\mathcal{M}(G_{[n]\setminus S'})$ is also dependent in
$\mathcal{M}(\mathbf{w})$
for every $\mathbf{w}\in U(S')$, hence also for every $\mathbf{w}\in V_{S'}$. 
In particular, since $\vv\in U(S)\subset V_{S'}$, it follows that every dependent set in $\mathcal{M}(G_{[n]\setminus S'})$ is also dependent in $\mathcal{M}(\mathbf{v})$.
Moreover, $\{i_{1},\ldots,i_{d}\}\subset N_{G}(i)\cap ([n]\setminus S')$, and since $i\in [n]\setminus S'$, this $d$-element set is dependent in $\mathcal{M}(G_{[n]\setminus S'})$. 
Hence, it follows that $\{v_{i_{1}},\ldots,v_{i_{d}}\}$ must be linearly dependent, contradicting their independence.

This contradiction shows that $V_{S}$ is irredundant, completing the proof.
\end{proof}

\subsection{Irreducible decomposition of $\OR$}

We are now in a position to describe the irreducible decomposition of $\OR$ when $G$ is a forest. 

\begin{theorem}\label{main theorem}
Let $G$ be a forest on $[n]$. The irreducible decomposition of $\OR$ is
\begin{equation}\label{irreducible irredundant}
\OR = \bigcup_{S} V_{S},
\end{equation}
where the union is taken over all $G$-admissible subsets of $[n]$.
\end{theorem}

\begin{proof}
The result follows directly from Propositions~\ref{prop: redundant decomposition} and~\ref{prop: maximal components}.
\end{proof}

\begin{example}
Consider the forest graph $G$ from Example~\ref{example: running example} (see Figure~\ref{forest graph}) and let $d = 4$. 
As observed in Example~\ref{example: admissible subsets}, the $G$-admissible subsets are exactly
\[
\bigl\{\emptyset,\, \{2\},\, \{3\},\, \{5\},\, \{3,5\}\bigr\}.
\]
It then follows from Theorem~\ref{main theorem} that $\4$ has precisely five irreducible components.
\end{example}

We conclude this section by recovering a result previously established in \textup{\cite[Theorem~1.5]{conca2019lovasz}}.

\begin{corollary}
Let $G$ be a forest, and let $\Delta(G)$ denote the maximum degree of a vertex of $G$. 
If $d \geq \Delta(G)+1$, then the variety $\OR$ is irreducible.
\end{corollary}

\begin{proof}
From the characterization of $G$-admissible subsets given in~\eqref{admissible description}, 
it follows that under the assumption $d \geq \Delta(G)+1$ the only $G$-admissible subset is $\emptyset$. 
Applying Theorem~\ref{main theorem} to this case yields the irreducibility of $\OR$.
\end{proof}

\section{\texorpdfstring{Dimension of the varieties $V_{S}$ and $\OR$}{Irreducible decomposition of OR for forest graphs}}\label{section 4}

In this section, we determine the dimension of $\OR$ and of its irreducible components.  
For this, we will make use of the following lemma.

\begin{lemma}\label{dim UG}
Let $G$ be a forest on $[n]$. Then
\[
\dim(U_{G}) = dn - |E(G)|.
\]
\end{lemma}

\begin{proof}
We argue by induction on the number of vertices $n$. The claim is clear when $n=1$. 
Assume it holds for all forests on at most $n-1$ vertices, and let $G$ be a forest on $[n]$.

Since $G$ is a forest, there exists a vertex $i \in [n]$ of degree at most one. 
The induced subgraph $G_{[n]\setminus \{i\}}$ is again a forest, now on fewer vertices. 
By the inductive hypothesis,
\begin{equation}\label{induc hypothesis}
\dim\bigl(U_{G_{[n]\setminus \{i\}}}\bigr)
= d(n-1) - |E(G_{[n]\setminus \{i\}})|
= dn - |E(G)| - (d-\deg(i)).
\end{equation}

We distinguish two cases:

\medskip
\noindent \textbf{Case \(\deg(i)=0\).} 
In this case, using the isomorphism in Equation~\eqref{isomorphism}, we obtain
\[
\dim(U_{G}) = \dim(U_{G_{[n]\setminus \{i\}}}) + d,
\]
and the result follows from \eqref{induc hypothesis}.

\medskip
\noindent \textbf{Case \(\deg(i)=1\).} 
Here, the $(\CC^{d-1})^{\ast}$-fibration described in Equation~\eqref{fibration} implies that
\[
\dim(U_{G}) = \dim(U_{G_{[n]\setminus \{i\}}}) + (d-1),
\]
which, combined with \eqref{induc hypothesis}, gives the desired formula.
\end{proof}

We now present the dimensions of the irreducible components of $\OR$.

\begin{theorem}\label{thm: dim of components}
Let $S$ be a $G$-admissible subset of $[n]$. Then
\[
\dim(V_{S}) = d(n-\size{S}) - |E(G_{[n]\setminus S})|.
\]
\end{theorem}

\begin{proof}
By definition, we have have that $U(S)\cong U_{G_{[n]\setminus S}}$. Applying Lemma~\ref{dim UG}, we obtain 
\[\dim(V_{S})=\dim(\closure{U(S)})=\dim(U(S))=\dim(U_{G_{[n]\setminus S}})=d(n-\size{S}) - |E(G_{[n]\setminus S})|.\]
\end{proof}

\begin{corollary}
The dimension of $\OR$ is given by
\[
\dim(\OR)
=\max\!\left\{\, d\bigl(n-\lvert S\bigr\rvert\bigr)\;-\;\lvert E(G_{[n]\setminus S})\rvert 
\;:\; S \text{ is $G$–admissible} \right\}.
\]
\end{corollary}

\begin{example}
Consider the forest graph $G$ from Example~\ref{example: running example} (see Figure~\ref{forest graph}) and let $d = 4$. 
The dimensions of the irreducible components of $\4$ are
\[
\dim(V_{\emptyset}) = \dim(V_{\{2\}}) = \dim(V_{\{3\}}) = \dim(V_{\{5\}}) = \dim(V_{\{3,5\}}) = 34.
\]
It follows that $\dim(\4) = 34$.
\end{example}

\section{\texorpdfstring{Primary decomposition of $L_{G}(d)$}{k}}\label{section 5}

In this section, we determine the primary decomposition of $L_{G}(d)$ for $G$ a forest graph. 
We begin by computing the ideals corresponding to the irreducible components $V_{S}$, which then leads to the primary decomposition.

\subsection{Ideals of the varieties $V_{S}$}

In this subsection, we focus on determining the ideals $\mathbb{I}(V_{S})$. Recall that these ideals are contained in the coordinate ring $\CC[X]$, where $X = (x_{i,j})$ is an $n \times d$ matrix of indeterminates, as in Definition~\ref{def: LSS ideals}.
Since $U(S)\cong U_{G_{[n]\setminus S}}$ and $\mathbb{I}(V_{S})=\mathbb{I}(U_{S})$, this reduces to understanding the ideal $\mathbb{I}(U_{G})$. We denote by $F_G$ the subvariety of $U_G$ consisting of all tuples $\vv = (v_1, \ldots, v_n) \in (\CC^d)^n$ satisfying:
\begin{itemize}
    \item $\vv$ is a faithful orthogonal representation of $\overline{G}$;
    \item $\vv$ realizes the matroid $\mathcal{M}(G)$.
\end{itemize}

By Lemma~\ref{lem: is dense}, the subvariety $F_G$ is dense in $U_G$, that is, $\overline{F_G} = \overline{U_G}$. In particular, the corresponding ideals coincide:
\[
\mathbb{I}(U_G) = \mathbb{I}(F_G).
\]

\begin{remark}\label{tñ}
Let $G$ be a forest on $[n]$, and fix a vertex $i \in [n]$. Choose distinct $d-1$ neighbors $i_1, \ldots, i_{d-1} \in N_G(i)$. For any $\vv = (v_1, \ldots, v_n) \in F_G$, we have
\[
v_i \perp \{v_{i_1}, \ldots, v_{i_{d-1}}\}.
\]
Since the vectors $v_{i_1}, \ldots, v_{i_{d-1}}$ are linearly independent and $v_i \neq 0$, it follows that $v_i$ is, up to a nonzero scalar, the normal vector to the hyperplane spanned by $v_{i_1}, \ldots, v_{i_{d-1}}$. Equivalently, if
\[
A = (v_{i_1} \mid \cdots \mid v_{i_{d-1}}),
\]
then, up to a nonzero scalar, the coordinates of $v_i$ are given by the signed maximal minors of $A$:
\[
(v_i)_j = (-1)^j \det(A_{[d]\setminus\{j\}}), \qquad j = 1, \ldots, d.
\]
This relation translates into an algebraic condition in $\mathbb{I}(F_G)$. Specifically, if $P \in \mathbb{I}(F_G) \subset \CC[X]$, then the polynomial obtained from $P$ by replacing the variables $x_{i,1}, \ldots, x_{i,d}$ with the signed maximal minors of the matrix
\[
X_{\{i_1, \ldots, i_{d-1}\}}
\]
also belongs to $\mathbb{I}(F_G)$. Here $X_{\{i_1, \ldots, i_{d-1}\}}$ denotes the submatrix of $X$ with rows indexed by $i_{1},\ldots,i_{d-1}$.
\end{remark}

\begin{example}
Consider the forest graph $G$ from Figure~\ref{forest graph} and let $d = 4$. 
Since $\{6,7,8\}$ are neighbors of $3$, we obtain the following. 
If $P \in \mathbb{I}(F_G) \subset \mathbb{C}[X]$, then the polynomial obtained from $P$ by replacing the variables $x_{3,1}, x_{3,2}, x_{3,3}, x_{3,4}$ with the signed maximal minors
{\scriptsize
\begin{equation}\label{signed maximal minors}
-\det
\begin{pmatrix} x_{6,2} & x_{6,3} & x_{6,4} \\ x_{7,2} & x_{7,3} & x_{7,4} \\ x_{8,2} & x_{8,3} & x_{8,4} \end{pmatrix}
\quad
\det
\begin{pmatrix} x_{6,1} & x_{6,3} & x_{6,4} \\ x_{7,1} & x_{7,3} & x_{7,4} \\ x_{8,1} & x_{8,3} & x_{8,4} \end{pmatrix}
\quad
-\det
\begin{pmatrix} x_{6,1} & x_{6,2} & x_{6,4} \\ x_{7,1} & x_{7,2} & x_{7,4} \\ x_{8,1} & x_{8,2} & x_{8,4} \end{pmatrix}
\quad
\det
\begin{pmatrix} x_{6,1} & x_{6,2} & x_{6,3} \\ x_{7,1} & x_{7,2} & x_{7,3} \\ x_{8,1} & x_{8,2} & x_{8,3} \end{pmatrix}
\end{equation}
}
also belongs to $\mathbb{I}(F_G)$.
\end{example}

Using the previous remark, we now construct an ideal $I_{G}$ associated with any forest~$G$.

\begin{definition}\label{gm}
Let $G$ be a forest. We define the ideal $I_G$, constructed as follows:
\begin{itemize}
    \item Define $X_0$ to be the set of polynomials generating the ideal $L_G(d)$, namely the polynomials $f_e$ from~\eqref{polynomials fe} for $e\in E(G)$. 
    For $j \ge 1$, define recursively the set $X_j$ to consist of all polynomials obtained from those in $X_{j-1}$ by modifying some of their variables, possibly leaving them unchanged, according to the substitution procedure described in Remark~\ref{tñ}. 
    Thus, we obtain an ascending chain
    \[
    X_0 \subseteq X_1 \subseteq X_2 \subseteq \cdots.
    \]
    \item For each $j \ge 0$, let $I_j$ denote the ideal generated by the polynomials in $X_j$. Then
    \[
    I_0 \subseteq I_1 \subseteq I_2 \subseteq \cdots.
    \]
    Since the polynomial ring $\CC[X]$ is Noetherian, this chain stabilizes. We denote the stable ideal by $I_G$.
\end{itemize}
\end{definition}

\begin{example}
Let $G$ be the forest graph from Figure~\ref{forest graph} and let $d = 4$. 
We have $f_{\{1,2\}} \in I_{0}$. 
Since $\{3,4,5\}$ are neighbors of $2$, it follows that
\[
\det(X_{\{1,3,4,5\}}) \in I_{1}.
\] 
Furthermore, as $\{6,7,8\}$ are neighbors of $3$, the polynomial obtained from $\det(X_{\{1,3,4,5\}})$ by replacing the variables $x_{3,1}, x_{3,2}, x_{3,3}, x_{3,4}$ with the signed maximal minors from~\eqref{signed maximal minors} belongs to $I_{2}$. 
Finally, all of these polynomials lie in $I_{G}$.
\end{example}

\begin{remark}\label{rem ñ}
If $P \in I_{G}$, then any polynomial obtained from $P$ by modifying some of its variables according to the procedure described in Remark~\ref{tñ} also belongs to $I_{G}$.
\end{remark}

It follows from Remark~\ref{tñ} and an induction argument that 
\begin{equation}\label{inclusion}
I_G \subseteq \mathbb{I}(F_G).
\end{equation}
Equivalently, $I_G$ is generated by all polynomials obtained through the iterative process of replacing the $d$ variables corresponding to a given row of $X$ with the signed Plücker minors of the submatrix of $X$ determined by any $d-1$ of its neighboring vertices.

We now show that the ideal $I_G$ provides, up to radical, a complete set of defining equations for the Zariski closure of $U_G$.

\begin{theorem}
\label{thm:IG-defines-UG}
For any forest $G$, we have
\begin{equation}\label{equations for UG}
\overline{U_G} \;=\; \mathbb{V}(I_G).
\end{equation}
In particular, $\mathbb{I}(U_G) = \sqrt{I_G}$.
\end{theorem}

\begin{proof}
By~\eqref{inclusion}, we have $F_G \subseteq \mathbb{V}(I_G)$. Since $F_G$ is dense in $U_G$, it follows that
\[
\overline{U_G}=\overline{F_G} \subseteq \mathbb{V}(I_G),
\]
which proves the inclusion ``$\subset$'' in~\eqref{equations for UG}.

For the reverse inclusion, let $\vv=(v_1,\ldots,v_n)\in\mathbb{V}(I_G)$; we must show that $\vv \in \overline{U_G}$. We prove that $\vv$ lies in the Euclidean closure of $U_G$: for every $\epsilon>0$ we construct $\tilde{\vv}\in U_G$ with $\|\vv-\tilde{\vv}\|<\epsilon$. Let
\[
\vv(0)=\{\,i\in[n]: v_i=0\,\}.
\]
We argue by induction on $m = |\vv(0)|$. If $m=0$, then $\vv\in U_G$ by definition. For the inductive step, assume the result holds for all tuples in $\mathbb{V}(I_G)$ with fewer than $m$ zero vectors and we prove it for $\vv\in \mathbb{V}(I_G)$ with $\size{\vv(0)}=m$.

\smallskip 

{\bf Step~1.} 
Suppose there exists $i\in\vv(0)$ such that
\[
\rank\{v_j : j\in N_G(i)\} = d-1.
\]
Choose $i_1,\ldots,i_{d-1}\in N_G(i)$ such that $v_{i_{1}},\ldots,v_{i_{d-1}}$ are linearly independent set.  
Define $\vv'=(v'_1,\ldots,v'_n)$ by $v'_j=v_j$ for $j\neq i$, and let $v'_i$ be a nonzero scalar multiple of the normal vector to the hyperplane spanned by $v_{i_1},\ldots,v_{i_{d-1}}$, chosen so that
\[
\|\vv'-\vv\|<\epsilon/2.
\]

We claim $\vv'\in\mathbb{V}(I_G)$. Indeed, for any $P\in I_G$,
\[
P(\vv') = \overline{P}(\vv),
\]
where $\overline{P}$ is obtained from $P$ by replacing the variables $x_{i,1},\ldots,x_{i,d}$ with the signed maximal minors of the matrix
\[
X_{\{i_1,\ldots,i_{d-1}\}}.
\]
Because $I_G$ is closed under this substitution (Definition~\ref{gm}), we have $\overline{P}\in I_G$, and thus $P(\vv')=0$.  
Hence $\vv'\in\mathbb{V}(I_G)$.

Since now $|\vv'(0)| = m-1$, the inductive hypothesis gives $\tilde{\vv}\in U_G$ with $\|\tilde{\vv}-\vv'\|<\epsilon/2$. Therefore
\[
\|\tilde{\vv}-\vv\| \le \|\vv-\vv'\| + \|\vv'-\tilde{\vv}\| < \epsilon,
\]
as required.

\smallskip

{\bf Step 2.}   
Assume now that
\begin{equation}\label{rank is at most d-2}
\rank\{v_j : j\in N_G(i)\}\le d-2 \qquad\text{for all } i\in\vv(0).
\end{equation}
Because $G$ is a forest, we may order the vertices of $\vv(0)$ as $x_1,\ldots,x_m$ so that each $x_i$ is adjacent to at most one vertex among $\{x_1,\ldots,x_{i-1}\}$. We construct elements
\[
\vv'_1,\vv'_2,\ldots,\vv'_m \in \OR,
\]
where each $\vv'_{i}$ is a tuple $\vv'_{i}=(v_{i,1},\ldots,v_{i,n})\in (\CC^{d})^{n}$
such that
\[
\|\vv-\vv'_1\| < \frac{\epsilon}{m}
\qquad\text{and}\qquad
\|\vv'_{i+1}-\vv'_i\| < \frac{\epsilon}{m}.
\]

For $\vv'_1$, keep $v_{1,j}=v_j$ for $j\neq x_1$ and choose a nonzero  
\[
v_{1,x_1}\in \operatorname{span}\{v_j : j\in N_G(x_1)\}^{\perp},
\]
which exists by~\eqref{rank is at most d-2}, and choose it so that  
$\|\vv'_1-\vv\|<\epsilon/m$. Note that $\vv'_{1}\in \OR$.

Inductively, define $\vv'_i$ by modifying only the coordinate at $x_i$, choosing
\[
v_{i,x_i}\in \operatorname{span}\{v_{i-1,j}: j\in N_G(x_i)\}^{\perp}\setminus\{0\},
\]
which is possible because each $x_i$ has at most one neighbor among $\{x_1,\ldots,x_{i-1}\}$, and~\eqref{rank is at most d-2} ensures the relevant span has dimension at most $d-1$. Again choose $v_{i,x_i}$ so that  
$\|\vv'_{i}-\vv'_{i-1}\|<\epsilon/m$.

After $m$ steps we obtain $\vv'_m\in\OR$ with all entries nonzero, hence $\vv'_m\in U_G$. Moreover,
\[
\|\vv-\vv'_m\|
\le \|\vv-\vv'_1\| + \sum_{i=1}^{m-1}\|\vv'_{i+1}-\vv'_{i}\|
< \frac{\epsilon}{m}\cdot m
= \epsilon.
\]

Thus $\vv\in\overline{U_G}$, completing the proof.
\end{proof}

We now describe the defining ideals of the varieties $V_{S}$.

\begin{theorem}\label{thm: ideal of VS}
For any subset $S \subseteq [n]$, the ideal of $V_{S}$ is
\[
\mathbb{I}(V_{S}) \;=\; (x_{i,j} : i \in S, j\in [d]) \;+\; \sqrt{I_{G_{[n]\setminus S}}}.
\]
\end{theorem}

\begin{proof}
This follows immediately from Theorem~\ref{thm:IG-defines-UG}, together with the observation that the restriction of $U(S)$ to the coordinates indexed by $[n]\setminus S$ is isomorphic to $U_{G_{[n]\setminus S}}$.
\end{proof}

\subsection{Primary decomposition of $L_{G}(d)$}

We now describe the primary decomposition of $L_{G}(d)$, which reflects the algebraic counterpart of the decomposition of $\OR$.  
Recall that our ideals live in the coordinate ring $\CC[X]$, where $X = (x_{i,j})$ is an $n \times d$ matrix of indeterminates, as introduced in Definition~\ref{def: LSS ideals}.

\begin{theorem}\label{thm: primary decomposition}
Let $G$ be a forest. The LSS ideal $L_{G}(d)$ has the primary decomposition
\[
L_{G}(d) \;=\; \bigcap_{S} \ \bigl( (x_{i,j} : i \in S, j\in [d]) \;+\; \sqrt{I_{G_{[n]\setminus S}}} \bigr),
\]
where the intersection ranges over all $G$-admissible subsets of $[n]$.
\end{theorem}

\begin{proof}
The result follows by passing to ideals in the decomposition of
Equation~\eqref{irreducible irredundant}, together with the description of 
$\mathbb{I}(V_{S})$ given in Theorem~\ref{thm: ideal of VS}, and using that 
$L_{G}(d)$ is radical; see \textup{\cite[Theorem~1.5]{conca2019lovasz}}.
\end{proof}

\begin{example}
Consider the forest graph $G$ from Figure~\ref{forest graph} and let $d = 4$. 
As observed in Example~\ref{example: admissible subsets}, the $G$-admissible subsets are 
\[
\bigl\{\emptyset,\, \{2\},\, \{3\},\, \{5\},\, \{3,5\}\bigr\}.
\]
Hence, the primary decomposition of $L_{G}(4)$ is
\[
\begin{aligned}
L_G(4) =\; & \sqrt{I_G} \;\cap\; 
\bigl((x_{2,j} : j \in [4]) + \sqrt{I_{G \setminus \{2\}}}\bigr) \;\cap\;
\bigl((x_{3,j} : j \in [4]) + \sqrt{I_{G \setminus \{3\}}}\bigr) \\[2mm]
& \cap \;\bigl((x_{5,j} : j \in [4]) + \sqrt{I_{G \setminus \{5\}}}\bigr) \;\cap\;
\bigl((x_{3,j}, x_{5,j} : j \in [4]) + \sqrt{I_{G \setminus \{3,5\}}}\bigr).
\end{aligned}
\]
\end{example}

\section{Examples}\label{section 6}

In this section, we apply the results from the previous sections, particularly Theorems~\ref{main theorem} and~\ref{thm: dim of components}, to several classical families of forest graphs, including \emph{star graphs}, \emph{caterpillar trees}, and \emph{binary trees}.

\medskip
\textbf{Star graphs.}
A \emph{star graph} $G$ is the bipartite graph $K_{1,n-1}$ on vertices $\{1,\ldots,n\}$ with edge set 
\[
\{\{1,i\} : 2 \le i \le n\},
\]
see Figure~\ref{star graph}. To determine the irreducible decomposition of $\OR$, we must identify all $G$-admissible subsets $S \subseteq [n]$. Recall that $S$ is $G$-admissible if and only if for every $i \in S$,
\begin{equation}\label{condition admissible}
\bigl|N_G(i)\cap([n]\setminus S)\bigr| \ge d
\quad\text{and}\quad
N_G(i)\cap([n]\setminus S)\not\subset
N_G(j)\cap([n]\setminus S)
\ \text{for every } j\in [n]\setminus S.
\end{equation}
Since vertex $1$ is the only vertex of degree greater than one, the only possible $G$-admissible sets are $S=\emptyset$ and $S=\{1\}$. Moreover, by~\eqref{condition admissible}, the set $S=\{1\}$ is admissible only when
\[
|N_G(1)| = n-1 \;\ge\; d.
\]
Applying Theorem~\ref{main theorem}, we obtain:
\begin{itemize}
\item $\OR$ is irreducible if $d>n-1$;
\item $\OR$ has two irreducible components, $V_{\emptyset}$ and $V_{\{1\}}$, if $n-1 \ge d$.
\end{itemize}

Moreover, by Theorem~\ref{thm: dim of components}, we obtain
\[
\dim(V_{\emptyset}) = dn - (n-1)
\qquad\text{and}\qquad
\dim(V_{\{1\}}) = dn - d.
\]
Consequently,
\[
\dim(\OR) \;=\; \max\{\, dn - (n-1),\; dn - d \,\}.
\]

\begin{figure}[H]
    \centering
    \includegraphics[width=0.45\textwidth]{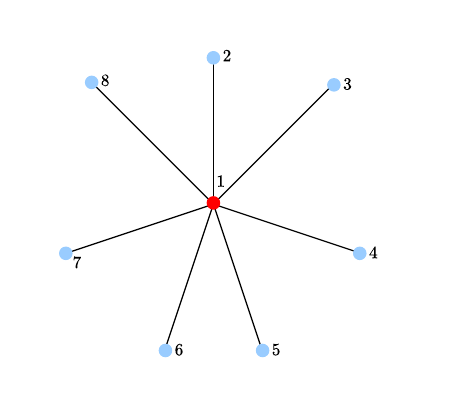}
    \caption{The star graph $K_{1,7}$.}
    \label{star graph}
\end{figure}

\begin{example}
Consider the star graph $G = K_{1,7}$ shown in Figure~\ref{star graph}.  
If $d>7$, then $\OR$ is irreducible;  
if $d \le 7$, then $\OR$ has exactly two irreducible components and 
\[\dim(\OR)=\max\{8d-7,7d\}.\]
\end{example}

\medskip
\textbf{Caterpillar trees.}
A \emph{caterpillar tree} is a tree in which every vertex lies at distance at most one from a designated central path, see Figure~\ref{caterpillar graph}.  
Let $G$ be a caterpillar tree on $[n]$ whose central path is the set $\{1,\ldots,k\}$.  
Every vertex $i\in\{k+1,\ldots,n\}$ is a leaf, adjacent to exactly one vertex of $[k]$.  
For each $i\in [k]$, define
\[
t_i \;=\; \bigl|N_G(i)\setminus [k]\bigr|.
\]
Then the caterpillar tree is completely determined (up to isomorphism) by the data
\[
k,\qquad t_1,\ldots,t_k.
\]
For example, for the caterpillar tree in Figure~\ref{caterpillar graph}, we have $k=9$ and 
$\{t_{1},\ldots,t_{9}\}=\{0,4,0,2,1,2,3,3,0\}$. To describe the irreducible decomposition of $\OR$, we must determine all $G$-admissible subsets $S\subseteq [n]$, that is, those satisfying~\eqref{condition admissible}.  
Since $d\ge 3$ and the only vertices of degree at least $3$ are those in $[k]$, necessarily
\[
S \subseteq [k].
\]
Furthermore, because each vertex in $\{k+1,\ldots,n\}$ is a leaf adjacent to a unique vertex of $[k]$, the admissibility condition~\eqref{condition admissible} simplifies to
\[
\bigl|N_G(i)\cap([n]\setminus S)\bigr| \;\ge\; d \qquad\text{for all } i\in S.
\]
For such $i\in S\subseteq [k]$ we may decompose
\[
N_G(i)\cap([n]\setminus S)
 = \bigl(N_G(i)\setminus [k]\bigr)
   \,\amalg\, \bigl(N_G(i)\cap([k]\setminus S)\bigr),
\]
and hence
\[
\bigl|N_G(i)\cap([n]\setminus S)\bigr|
  \;=\; t_i + \bigl|N_G(i)\cap([k]\setminus S)\bigr|.
\]

By Theorem~\ref{main theorem}, irreducible components of $\OR$ correspond exactly to those subsets $S\subseteq [k]$ such that
\begin{equation}\label{condition for admsis caterpillar}
t_i + \bigl|N_G(i)\cap([k]\setminus S)\bigr| \;\ge\; d
\qquad\text{for all } i\in S.
\end{equation}
This condition may be interpreted as requiring that, for each $i\in S$, at least one of the following holds:
\begin{enumerate}[label=(\arabic*), ref=(\arabic*)]
\item \label{cond 1} $t_i \ge d$,
\item \label{cond 2} $t_i = d-1$ and $i$ has at least one neighbor in $[k]\setminus S$,
\item \label{cond 3} $t_i = d-2$ and $i$ has exactly two neighbors in $[k]\setminus S$.
\end{enumerate}

Moreover, for any such $S$, Theorem~\ref{thm: dim of components} yields
\begin{equation}\label{dim for caterpillar}
\dim(V_S)
 \;=\;
 d(n-|S|)
 \;-\;
 |E(G_{[n]\setminus S})|
 \;=\;
 d(n-|S|)
 \;-\;
 |E(G_{[k]\setminus S})|
 \;-\;
 \sum_{i\in [k]\setminus S} t_i.
\end{equation}

\begin{figure}[H]
    \centering
    \includegraphics[width=0.9\textwidth]{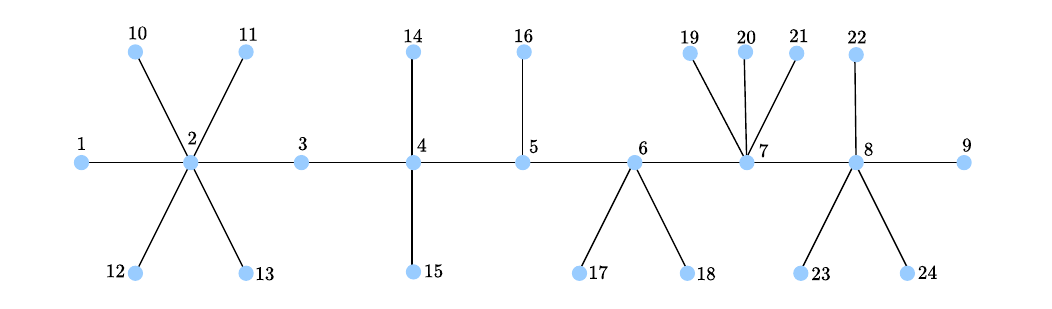}
    \caption{A caterpillar tree.}
    \label{caterpillar graph}
\end{figure}

\begin{example}
Consider the caterpillar tree $G$ from Figure~\ref{caterpillar graph} with vertex set $[24]$, and let $d=5$.  
Here the central path has length $k=9$ and 
\[
(t_1,\ldots,t_9)=(0,4,0,2,1,2,3,3,0).
\]
Using conditions~\ref{cond 1}–\ref{cond 3}, we find that $S$ must be a subset of $\{2,7,8\}$, since all other $i\in [9]$ satisfy $t_i<3=d-2$.  
Furthermore, $S$ cannot contain both $7$ and $8$, because each would then require two neighbors in $[k]\setminus S$, which is impossible.

Thus, the $G$-admissible subsets are
\[
\{\emptyset,\ \{7\},\ \{8\},\ \{2\},\ \{2,7\},\ \{2,8\}\},
\]
and so $\OR$ has exactly $6$ irreducible components.  
Finally, using~\eqref{dim for caterpillar}, these components have dimensions
\[
\begin{aligned}
&\dim(V_{\emptyset})=97,\qquad
 \dim(V_{\{7\}})=97,\qquad
 \dim(V_{\{8\}})=97,\\[2pt]
&\dim(V_{\{2\}})=98,\qquad
 \dim(V_{\{2,7\}})=98,\qquad
 \dim(V_{\{2,8\}})=98.
\end{aligned}
\]
\end{example}

\medskip
\textbf{Binary trees.}
A {\em binary tree} is a rooted tree in which each vertex has at most two children; see Figure~\ref{binary tree}.  
Let $G$ be a binary tree on $[n]$ and let $d\ge 3$.  
To describe the irreducible decomposition of $\OR$, we must determine all $G$-admissible subsets $S\subseteq [n]$, that is, those satisfying~\eqref{condition admissible}.  
Since every vertex of $G$ has degree at most three, if $d\ge 4$ then no vertex satisfies the admissibility condition, so the only admissible subset is $S=\emptyset$.  
By Theorem~\ref{main theorem}, it follows that $\OR$ is irreducible in this case.

We now consider the case $d=3$.  
Here, any admissible subset $S$ must be contained in the set of vertices of degree three.  
Moreover, for such an $S$, it is straightforward to check that~\eqref{condition admissible} is equivalent to requiring that
\[
N_{G}(i)\cap S=\emptyset 
\qquad\text{for all } i\in S.
\]
Thus, by Theorem~\ref{main theorem}, the irreducible components of $\3$ are in bijection with all subsets $S$ of vertices of degree three in which no two vertices are adjacent.  
For any such $S$, Theorem~\ref{thm: dim of components} gives
\begin{equation}\label{dim for binary tree}
\dim(V_S)
 \;=\;
 3(n-|S|)
 \;-\;
 |E(G_{[n]\setminus S})|
 \;=\;
 3(n-|S|)-(|E(G)|-3|S|)
 \;=\;
 2n+1,
\end{equation}
where the second equality uses that no two vertices of $S$ are adjacent, and the last equality uses $|E(G)|=n-1$.  
Hence, $\3$ is equidimensional, and every irreducible component has dimension $2n+1$.

\begin{figure}[H]
    \centering
    \includegraphics[width=0.5\textwidth]{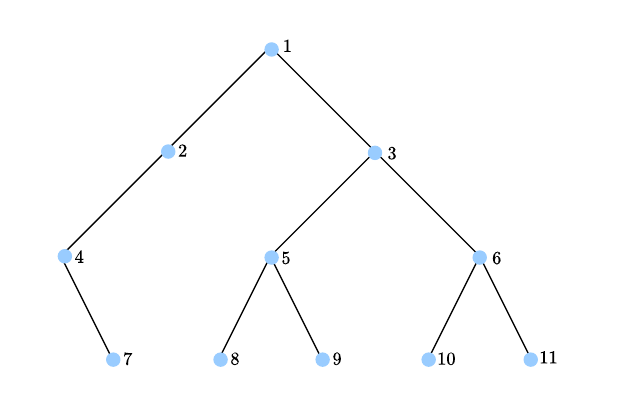}
    \caption{A binary tree.}
    \label{binary tree}
\end{figure}

\begin{example}
Consider the binary tree $G$ from Figure~\ref{binary tree} with vertex set $[11]$, and let $d=3$.  
The vertices of degree~$3$ are $\{3,5,6\}$.  
Hence, the $G$-admissible subsets are
\[
\{\emptyset,\{5\},\{6\},\{5,6\},\{3\}\},
\]
and thus $\3$ has exactly five irreducible components.  
All of them have dimension $23$.
\end{example}

\bibliographystyle{abbrv}
\bibliography{Citation}



\medskip
{\footnotesize\noindent {\bf Author's addresses}
\medskip

\noindent{Emiliano Liwski, 
KU Leuven}\hfill {\tt  emiliano.liwski@kuleuven.be}

\end{document}